\documentclass[11pt]{article}
\usepackage{amssymb}
\usepackage{amsmath}
\usepackage{amsthm} % defines environment proof
\usepackage{wrapfig}
\usepackage{times}
\usepackage{color}

\usepackage{graphicx}
\usepackage{float} % umoznuje obrazky
\usepackage{epic}  % umoznuje \dashline, \drawline \spline atd.

\allowdisplaybreaks[4] % The optional parameter 1-4 sets the intensity
\newtheorem{lemma}{Lemma}%[section]
\newtheorem{remark}{Remark}[section]
\newtheorem{theorem}{Theorem}%[section]

\renewenvironment{proof}{ %\removelastskip\par\medskip
\noindent{\bf Proof.} \rm}{\penalty-20\null\hfill $\square$}
\numberwithin{equation}{section}

\newcommand{\N}{{\mathbb N}}
\newcommand{\R}{{\mathbb R}}

\newcommand{\Z}{{\mathbb Z}}

\newcommand{\bfn}{\mathbf{n}}

\newcommand{\bfu}{\mathbf{u}}
\newcommand{\bfv}{\mathbf{v}}
\newcommand{\bfw}{\mathbf{w}}
\newcommand{\bfx}{\mathbf{x}}

\newcommand{\bfzeta}{\mbox{\boldmath $\zeta$}}

\newcommand{\bfomega}{\mbox{\boldmath $\omega$}}

\newcommand{\bftau}{\mbox{\boldmath $\tau$}}

\newcommand{\bfC}{\mathbf{C}}

\newcommand{\bfL}{\mathbf{L}}

\newcommand{\bfV}{\mathbf{V}}
\newcommand{\bfW}{\mathbf{W}}

\newcommand{\rmd}{\mathrm{d}}

\renewcommand{\div}{\mathrm{div}\,}

\newcommand{\bfzero}{\mathbf{0}}

\newcommand{\supp}{\mathrm{supp}\,}

\newcommand{\curl}{\mathbf{curl}}

\newcommand{\br}{\hbox to 0.7pt{}}

\newcommand{\Ls}{\bfL^2_{\sigma}(\R^3)}
\newcommand{\Lsp}{\Ls^+}
\newcommand{\Lsm}{\Ls^-}
\newcommand{\Ws}{\bfW^{1,2}_{\sigma}(\R^3)}
\newcommand{\Wsdual}{\bfW^{-1,2}_{\sigma}(\R^3)}

\newcommand{\bfomegae}{\bfomega_{e}}
\newcommand{\bfve}{\bfv_{e}}

\newcommand{\BRR}{B_{R+1}}

\newcommand{\CRR}{C_{R+1}}

\newcommand{\Spe}{{\rm Sp}}

\newcounter{constants}
\newcommand{\cn}[2]{ \addtocounter{constants}{1}
\newcounter{c#1#2}
\setcounter{c#1#2}{\value{constants}} c_{\arabic{c#1#2}} }
\newcommand{\cc}[2]{c_{\arabic{c#1#2}}}

\begin{document}

\title{\LARGE \bf Regularity of a Weak Solution to the Navier--Stokes
Equations via One Component of a Spectral \\ Projection of
Vorticity}

\author{Ji\v{r}\'{\i} Neustupa \ and \ Patrick Penel}

\date{}

\maketitle

\begin{abstract}
We deal with a weak solution $\bfv$ to the Navier--Stokes initial
value problem in $\R^3\times(0,T)$. We denote by $\bfomega^+$ a
spectral projection of $\bfomega\equiv\curl\, \bfv$, defined by
means of the spectral resolution of identity associated with the
self--adjoint operator $\curl$. We show that certain conditions
imposed on $\bfomega^+$ or, alternatively, only on $\omega^+_3$
(the third component of $\bfomega^+$) imply regularity of solution
$\bfv$.
\end{abstract}

\noindent
{\it AMS math.~classification (2000):} \ 35Q30, 76D03, 76D05.

\noindent
{\it Keywords:} \ Navier--Stokes equations, weak solution,
regularity criteria, vorticity.

\section{Introduction} \label{S1}

{\bf The Navier--Stokes problem.} \ Let $T>0$. We denote
$Q_T:=\R^3\times(0,T)$. We deal with the Navier--Stokes initial
value problem
\begin{align}
\partial_t\bfv\, +\, &\bfv\cdot\nabla\bfv\ =\ -\nabla p+\nu\Delta\bfv &&
\mbox{in}\ Q_T, \label{1.1} \\ \noalign{\vskip 2pt}
& \div\bfv\ =\ 0 \qquad && \mbox{in}\ Q_T, \label{1.2} \\
\noalign{\vskip 2pt}
& \bfv(\bfx,\, .\, )\, \longrightarrow\, 0 && \mbox{for}\
|\bfx|\to\infty, \label{2a} \\ \noalign{\vskip 2pt}
& \bfv(\, .\, ,0)\ =\ \bfv_0 && \mbox{in}\ \R^3 \label{1.3}
\end{align}
for the unknown velocity $\bfv$ and pressure $p$. The symbol $\nu$
denotes the coefficient of viscosity. It is usually assumed to be
a positive constant. Since its value plays no role throughout the
paper, we assume that $\nu=1$.

We assume that $\bfv$ is a weak solution of the problem
(\ref{1.1})--(\ref{1.3}). (This notion was introduced by Leray
\cite{Le}, the exact definition is also given e.g.~in \cite{Ga2}.)
In accordance with \cite{CKN}, we define a {\it regular point} of
solution $\bfv$ as a point $(\bfx,t)\in Q_T$ such that there
exists a space--time neighbourhood of $(\bfx,t)$, where $\bfv$ is
essentially bounded. Points in $Q_T$ that are not regular are
called {\it singular.} The question whether a weak solution can
develop a singularity at some time instant $t_0\in(0,T]$ or if all
points $(\bfx,t)\in Q_T$ are regular points is an important open
problem in the theory of the Navier--Stokes equations. There exist
many a posteriori criteria, stating that if the weak solution has
certain additional properties then it has no singular points (in
the whole $Q_T$ or at least in a sub--domain of $Q_T$). The
studies of such criteria have been mainly motivated by Leray
\cite{Le} (who proved that if the weak solution belongs to the
class $L^r(0,T;\, \bfL^s(\R^3))$, where $3<s\leq\infty$ and \
$2/r+3/s=1$, then it is infinitely differentiable in $Q_T$) and by
Serrin \cite{Se} (who proved a certain analog of Leray's
criterion, applicable in a sub--domain of $Q_T$). Exact citations
and further details on this topic can be found in the survey paper
\cite{Ga2} by Galdi.

\vspace{4pt} \noindent
{\bf On some previous results.} \ Here, we focus on regularity
criteria that impose additional conditions on some components of
velocity $\bfv=(v_1,v_2,v_3)$ or its gradient $\nabla\bfv$ or the
corresponding vorticity $\bfomega=(\omega_1,\omega_2,\omega_3)$.

The first result on regularity as a consequence of an a posteriori
assumption on one velocity component appeared in \cite{NePe1}: the
authors considered the problem in a domain $\Omega\subset\R^3$,
assumed that the component $v_3$ is essentially bounded in a
space--time region $D\subset\Omega\times(0,T)$, and proved that
$\bfv$ has no singular points in $D$. This result has been later
successively improved in \cite{NeNoPe} ($v_3$ is only supposed to
be in $L^r(t_1,t_2;\, L^s(D'))$ for all sets
$D'\times(t_1,t_2)\subset D$ and some $r\in[4,\infty]$,
$s\in(6,\infty]$ satisfying $2/r+3/s\leq\frac{1}{2}$),
\cite{NePe2} (the authors generalize the result from \cite{NeNoPe}
and combine an assumption on $v_3$ with assumptions on $v_1$,
$v_2$), \cite{KuZi2} ($v_3$ is only assumed to be in $L^r(0,T;\,
L^s(\R^3))$, where $2/r+3/s=\frac{5}{8}$ for
$r\in[\frac{16}{5},\infty)$ and $s\in(\frac{24}{5},\infty]$),
\cite{CaTi} (the authors consider the spatially periodic problem
in $\R^3$ and use the condition $2/r+3/s<\frac{2}{3}+2/(3s)$,
$s>\frac{7}{2}$), and \cite{ZhPo1} (the exponents $r$, $s$ are
supposed to satisfy the conditions $2/r +3/s\leq\frac{3}{4}+
1/(2s)$, $s>\frac{10}{3}$).

Of a series of papers, where the authors deal with the question of
regularity of weak solution $\bfv$ in dependence on certain
integrability properties of some components of the tensor
$\nabla\bfv$, we mention \cite{BdV}, \cite{ChaCho}, \cite{KuZi1},
\cite{KuZi2}, \cite{ZhPo1}, \cite{ZhPo2} and \cite{PePo}. In paper
\cite{ChaCho}, the authors prove regularity of solution $\bfv$ by
means of conditions imposed on only two components of vorticity.
They assume that the initial velocity $\bfv_0$ is ``smooth'' and
$\omega_1$, $\omega_2\in L^r(0,T;\, L^s(\R^3))$ with $1<r<\infty$,
$\frac{3}{2}<s<\infty$, $2/r+3/s\leq 2$ or the norms of $\omega_1$
and $\omega_2$ in $L^{\infty}(0,T;\, L^{3/2}(\R^3))$ are
``sufficiently small''. It is a challenging open problem to show
whether regularity of weak solution $\bfv$ can be controlled by
only one component of vorticity.

The cited criteria that concern the solution in the whole space
hold for any weak solution, while the interior regularity criteria
hold for the so called suitable weak solution because here we need
to apply an appropriate localization procedure (see
e.g.~\cite{NePe2}, the concept of suitable weak solutions has been
introduced in \cite{CKN}).

The results mentioned above represent attempts to find a minimum
quantity which controls regularity of the solution. If such a
quantity is in some sense smooth or integrable then the weak
solution is smooth. On the other hand, each such quantity
necessarily loses smoothness if a singular point shows up. Thus,
the criteria contribute to understanding the behaviour of the
solution in the neighbourhood of a hypothetic singular point. The
presented paper brings results in this field. The quantity, which
is assumed to be ``smooth'' in this paper, is either a certain
spectral projection of vorticity or only one component of that
spectral projection. The projection is defined by means of the
spectral resolution of identity associated with operator $\curl$,
see (\ref{1.4}). In the case of only one component, we need to
impose a stronger condition on its regularity than in the case of
all three components, see Theorems \ref{T1} and \ref{T2}.

\vspace{4pt} \noindent
{\bf Notation and auxiliary results.} \ We denote vector functions
and spaces of such functions by boldface letters. The norm in
$L^q(\R^3)$ (or $\bfL^q(\R^3)$) is denoted by $\|\, .\, \|_{q;\,
\R^3}$. The scalar product in $\bfL^2(\R^3)$ is denoted by $(\,
.\, ,\, .\, )_{2;\, \R^3}$. The norm in $W^{s,q}(\R^3)$ (or
$\bfW^{s,q}(\R^3)$) is denoted by $\|\, .\, \|_{s,q;\, \R^3}$.
Other norms and scalar products are denoted by analogy.

The space $\Ls$ is a completion of
$\bfC^{\infty}_{0,\sigma}(\R^3)$ (the linear space of infinitely
differentiable divergence--free vector functions in $\R^3$ with a
compact support) in $\bfL^2(\Omega)$. The intersection
$\bfW^{1,2}(\R^3)\cap\Ls$ is denoted by $\Ws$. It is a closed
subspace of $\bfW^{1,2}(\R^3)$. Note that $\|\nabla\bfu\|_{2;\,
\R^3}^2=\|\curl\, \bfu\|_{2;\, \R^3}^2$ for $\bfu\in\Ws$.
Consequently, $\|\, .\, \|_{1,2;\, \R^3}^2=\|\, .\, \|_{2;\,
\R^3}^2+\|\curl\; .\, \|_{2;\, \R^3}^2$ in $\Ws$. The dual to
$\Ws$ is denoted by $\Wsdual$.

The operator $(-\Delta)$, with the domain $W^{2,2}(\R^3)$
(respectively $\bfW^{2,2}(\R^3)$), is positive and self--adjoint
in $L^2(\R^3)$ (respectively in $\bfL^2(\R^3)$). Its spectrum is
purely continuous and covers the non--negative part of the real
axis, see e.g.~\cite{FaNeNe}.

The Stokes operator $S:=\curl^2$, as an operator in $\Ls$,
coincides with the reduction of $(-\Delta)$ to $\Ls$, see
e.g.~\cite[p.~138]{So}. The domain of $S$ is the space
$\bfW^{2,2}(\R^3)\cap\Ls$. Operator $S$ is positive, and its
spectrum is continuous and covers the interval $[0,\infty)$ on the
real axis, see \cite{FaNe} or \cite{FaNeNe}. The power $S^{1/4}$
of operator $S$ satisfies the Sobolev--type inequality
$\|\bfu\|_{3;\, \R^3}\leq \cn01\, \|S^{1/4}\bfu\|_{2;\, \R^3}$ for
$\bfu\in D(S^{1/4})$, see \cite[p.~141]{So}.

\begin{lemma}
\label{L1}
Operator $\curl$, with the domain $D(\curl):=\Ws$, is
self--adjoint in $\Ls$. Its spectrum is continuous and coincides
with the whole real axis.
\end{lemma}

\begin{proof}
Operator $\curl$ maps $\Ws$ into $\Ls$. The symmetry of $\curl$
follows from an easy integration by parts. The symmetry means that
$\curl\subset\curl^*$, where $\curl^*$ is the adjoint operator to
$\curl$. In order to prove that $\curl=\curl^*$, it is sufficient
to show that $D(\curl^*)\subset\Ws$. Thus, let $\bfu\in
D(\curl^*)$. Then there exists $\bfu^*\in\Ls$ such that $(\curl\,
\bfv,\, \bfu)_{2;\, \R^3}=(\bfv,\bfu^*)_{2;\, \R^3}$ for all
$\bfv\in\Ws$. There exists a sequence $\{\bfu_n\}$ in $\Ws$,
converging to $\bfu$ in the norm of $\Ls$. For each $\bfv\in\Ws$,
we have
\begin{displaymath}
(\curl\, \bfv,\bfu)_{2;\, \R^3}\ =\ \lim_{n\to\infty}\, (\curl\,
\bfv,\, \bfu_n)_{2;\, \R^3}\ =\ \lim_{n\to\infty}\, (\bfv,\,
\curl\, \bfu_n)_{2;\, \R^3}.
\end{displaymath}
Thus, $(\bfv,\bfu^*)_{2;\, \R^3}=\lim_{n\to\infty}\, (\bfv,\,
\curl\, \bfu_n)_{2;\, \R^3}$. This holds, due to the density of
$\Ws$ in $\Ls$, for all $\bfv\in\Ls$. Hence the sequence
$\{\curl\, \bfu_n\}$ converges weakly to $\bfu^*$ in $\Ls$.
Furthermore, since $\curl$ maps continuously $\Ls$ to $\Wsdual$,
we also have $\curl\, \bfu_n\to\curl\, \bfu$ in $\Wsdual$. Hence
$\curl\, \bfu=\bfu^*\in\Ls$. This inclusion, together with the
fact that $\bfu\in\Ls$, implies that $\bfu\in\Ws$. We have proven
that operator $\curl$ is self--adjoint in $\Ls$.

The spectrum of $\curl$, which we denote by $\Spe(\curl)$, is a
subset of the real axis. The residual part is empty, because
$\curl$ is self--adjoint. It means that each point
$\lambda\in\Spe(\curl)$ is either an eigenvalue, or it belongs to
$\Spe_c(\curl)$ (the continuous spectrum of $\curl$). If $\lambda$
is an eigenvalue then $\lambda^2$ is an eigenvalue of the Stokes
operator $S$, which is impossible (see e.g.~\cite[Lemma
2.6]{FaNe}). Thus, $\Spe(\curl)=\Spe_c(\curl)$.

Let us finally show that the spectrum covers the whole real axis.
All points of $\Spe_c(\curl)$ are non--isolated, otherwise they
would have been the eigenvalues, see \cite[p.~273]{Ka}. Let
$\lambda\in\Spe_c(\curl)$, $\lambda\not=0$. There exists a
sequence $\{\bfu_n\}$ on the unit sphere in $\Ls$, such that
$\|\curl\, \bfu_n-\lambda\bfu_n\|_{2;\, \R^3}\to 0$. Let
$\xi\in\R$, $\xi\not=0$. Put $\bfu_n^{\xi}(\bfx):=\xi^{3/2}\,
\bfu_n(\xi\bfx)$. Then $\{\bfu_n^{\xi}\}$ is a sequence on the
unit sphere in $\Ls$, satisfying $\|\curl\,
\bfu_n^{\xi}-\xi\br\lambda\bfu_n^{\xi}\|_{2;\, \R^3}\to 0$. It
means that $\xi\lambda$ belongs to $\Spe_c(\curl)$ as well. Thus,
each real number, different from zero, is in $\Spe_c(\curl)$.
Since $\Spe_c(\curl)$ is closed, we obtain the equality
$\Spe_c(\curl)=\R$.
\end{proof}

\medskip
Let us note that a self--adjoint realization of operator $\curl$
in an exterior domain, in a more general framework than in the
space $\Ls$, has been studied in \cite{Pi}.

Let $\{E_{\lambda}\}$ be the spectral resolution of identity,
associated with operator $\curl$. Projection $E_{\lambda}$ is
strongly continuous in dependence on $\lambda$ because
$\Spe(\curl)$ is continuous, see \cite[pp.~353--356]{Ka}. We
denote
\vspace{-8pt}
\begin{equation}
P^-:=E_0=\int_{-\infty}^0\rmd E_{\lambda} \qquad \mbox{and} \qquad
P^+:=I-E_0=\int_0^{\infty}\rmd E_{\lambda}\br.
\label{1.4}
\end{equation}
Operators $P^-$ and $P^+$ are orthogonal projections in $\Ls$ such
that $I=P^-+P^+$ and $O=P^-P^+$. We put $\Lsm:=P^-\Ls$ and
$\Lsp:=P^+\Ls$. Both $\Lsm$ and $\Lsp$ are closed subspaces of
$\Ls$. Operator $\curl$ reduces on each of the spaces $\Lsm$ and
$\Lsp$. It is negative on $\Lsm$ and positive on $\Lsp$. We denote
by $A$ the operator $|\br\curl\br|$, i.e.
\vspace{-8pt}
\begin{equation}
A\ :=\ -\curl\, \bigl|_{\Lsm}+\, \curl\, \bigr|_{\Lsp}\br.
\label{1.5}
\end{equation}

\begin{lemma}
\label{L2}
Operator $A$ is positive, self--adjoint, and $A=S^{1/2}$.
\end{lemma}

\begin{proof}
Operator $A$ is self--adjoint and positive in each of the spaces
$\Lsm$ and $\Lsp$, hence it is self--adjoint and positive in $\Ls$
as well. (See also \cite[p.~358]{Ka}.)

The formula $A^2\bfu=S\bfu$ clearly holds for $\bfu\in D(A^2)\cap
D(S)$. Clearly, $D(S)\subset D(A^2)$. We claim that the opposite
inclusion $D(A^2)\subset D(S)$ is also true: \ the domain of $A^2$
is, by definition, the space of all $\bfu\in\Ws$ such that
$A\bfu\in\Ws$. Using the decomposition $\bfu=P^-\bfu+P^+\bfu$ and
the fact that both the operators $A$ and $\curl$ are reduced on
$\Ls^-$ and on $\Ls^+$, one can verify that $\bfu\in D(A^2)$
satisfies
\begin{displaymath}
\|A\bfu\|_{1,2;\, \R^3}^2\ =\ \|\curl\, A\bfu\|_{2;\,
\R^2}^2+\|A\bfu\|_{2;\, \R^3}^2\ =\ \|\curl^2\bfu\|_{2;\,
\R^3}^2+\|\curl\, \bfu\|_{2;\, \R^3}^2\ <\ \infty.
\end{displaymath}
This implies that $\bfu\in\bfW^{2,2}(\R^3)\cap\Ws=D(S)$, hence
$D(A^2)\subset D(S)$. Consequently, $A^2=S$.

The resolution of identity associated with operator $A$ is the
system of projections $F_{\lambda}:=O$ for $\lambda<0$, $\,
F_{\lambda}=E_{\lambda}-E_{-\lambda}$ for $\lambda>0$. The family
of projections $G_{\lambda}:=O$ for $\lambda<0$, $\,
G_{\lambda}:=F_{\sqrt{\lambda}}$ for $\lambda>0$, represents the
resolution of identity associated with the operator $A^2\equiv S$.
Operator $A$ can now be expressed in this way:
\begin{displaymath}
A\ =\ \int_0^{\infty}\lambda\; \rmd F_{\lambda}\ =\
\int_0^{\infty}\sqrt{\zeta}\; \rmd F_{\sqrt{\zeta}}\ =\
\int_0^{\infty}\sqrt{\zeta}\; \rmd G_{\zeta}\ =\ S^{1/2}.
\end{displaymath}
This completes the proof.
\end{proof}

\medskip
Another way, how one can prove the identity $A=S^{1/2}$, is the
application of Theorem 3.35 in \cite{Ka}. However, here one needs
to verify that both the operators $S$ and $A$ are m--accretive.
The identity $A=S^{1/2}$ also follows from \cite[Theorem 4,
p.~144]{BiSo}.

Due to Lemma \ref{L2}, $A^{\alpha}=S^{\alpha/2}$ for $\alpha\geq
0$. Consequently,
\begin{equation}
\|\bfu\|_{3;\, \R^3}\ \leq\ \cc01\, \|A^{1/2}\bfu\|_{2;\, \R^3}
\label{1.6}
\end{equation}
for $\bfu\in D(A^{1/2})$.

Recall that $\bfomega=\curl\, \bfv$. We further denote
$\bfv^-:=P^-\bfv$, $\, \bfv^+:=P^+\bfv$, $\,
\bfomega^-:=P^-\bfomega$ and $\bfomega^+:=P^+\bfomega$. The
components of $\bfv^+$ are denoted by $v_1^+$, $v_2^+$ and
$v_3^+$, the components of functions $\bfv^-$, $\bfomega^-$ and
$\bfomega^+$ are denoted by analogy. Since operator $\curl$
commutes with projections $P^-$ and $P^+$, we have
$\bfomega^-=\curl\, \bfv^-=-A\bfv^-$ and $\bfomega^+=\curl\,
\bfv^+=A\bfv^+$.

As a weak solution of the problem (\ref{1.1})--(\ref{1.3}), $\bfv$
belongs to $L^2(0,T;\, \Ws)\cap L^{\infty}(0,T$; $\Ls)$. We say
that $\bfv$ satisfies (EI) (= the energy inequality) if
\begin{equation}
\|\bfv(t)\|_{2;\, \R^3}^2+2\int_s^t\|\nabla\bfv(\tau)\|_{2;\,
\R^3}^2\; \rmd\tau\ \leq\ \|\bfv(s)\|_{2;\, \R^3}^2 \label{1.7}
\end{equation}
for $s=0$ and all $t\in[0,T)$. We say that $\bfv$ satisfies (SEI)
(= the strong energy inequality) if (\ref{1.7}) holds for
a.a.~$s\in[0,T)$ and all $t\in[s,T)$.

The next two theorems represent the main results of the paper.

\begin{theorem}
\label{T1}
Let $\bfv$ be a weak solution to the problem
(\ref{1.1})--(\ref{1.3}). Assume that at least one of the two
conditions

\begin{list}{}
{\setlength{\topsep 4pt}
\setlength{\itemsep 2pt}
\setlength{\leftmargin 50pt}
\setlength{\rightmargin 0pt}
\setlength{\labelwidth 25pt}}

\item[{\rm (i)} ]
$(-\Delta)^{1/4}\bfomega^+\in\bfL^2(Q_T)$,

\item[{\rm (ii)} ]
$(-\Delta)^{3/4}\omega_3^+\in L^2(Q_T)$

\end{list}
and at least one of the two conditions

\begin{list}{}
{\setlength{\topsep 4pt}
\setlength{\itemsep 2pt}
\setlength{\leftmargin 50pt}
\setlength{\rightmargin 0pt}
\setlength{\labelwidth 25pt}}

\item[{\rm (a)} ]
$\bfv_0\in\Ls$ and $\bfv$ satisfies (SEI),

\item[{\rm (b)} ]
$\bfv_0\in D(A^{1/2})$ and $\bfv$ satisfies (EI)

\end{list}

\noindent
hold. Then the norm $\|A^{1/2}\bfv\|_{2;\, \R^3}$ is bounded in
each time interval $(\vartheta,T)$, where $0<\vartheta<T$. (If
condition (b) holds then $\|A^{1/2}\bfv\|_{2;\, \R^3}$ is even
bounded on the whole interval $(0,T)$.) Consequently, solution
$\bfv$ has no singular points in $Q_T$.
\end{theorem}

The proof of existence of a weak solution to
(\ref{1.1})--(\ref{1.3}), satisfying (EI) and (SEI) under the
assumption that $\bfv_0\in\Ls$, is given in \cite{Le}. Thus,
conditions (a) and (b) do not cause any remarkable restrictions.

The next theorem is a generalization of Theorem \ref{T1}. Before
we formulate it, we introduce some notation. Suppose that $a=a(t)$
is a function in the interval $(0,T)$ with values in
$[-\infty,\infty)$. We denote by $a_+(t)$ the positive part and by
$a_-(t)$ the negative part of $a(t)$. We put
\begin{equation}
P^+_{a(t)}\ :=\ I-E_{a(t)}\ =\ \int_{a(t)}^{\infty}\rmd
E_{\lambda}\br,\label{G1}
\end{equation}
$\bfv^+_a(t):=P^+_{a(t)}\bfv(t)$, and
$\bfomega^+_a(t):=P^+_{a(t)}\bfomega(t)=\curl\, \bfv^+_a(t)$. The
third component of function $\bfomega^+_a$ is denoted by
$\omega^+_{a3}$.

\begin{theorem}
\label{T2}
Let $\bfv$ be a weak solution to the problem
(\ref{1.1})--(\ref{1.3}). Assume that at least one of the two
conditions

\begin{list}{}
{\setlength{\topsep 4pt}
\setlength{\itemsep 2pt}
\setlength{\leftmargin 50pt}
\setlength{\rightmargin 0pt}
\setlength{\labelwidth 25pt}}

\item[{\rm (iii)} ]
$a_+\in L^3(0,T)$ \ and \
$(-\Delta)^{1/4}\bfomega^+_a\in\bfL^2(Q_T)$,

\item[{\rm (iv)} ]
$a_+\in L^3(0,T)$, \ $a_-\in L^5(0,T)$ \ and \
$(-\Delta)^{3/4}\omega^+_{a3}\in L^2(Q_T)$

\end{list}

\noindent
and at least one of conditions (a) and (b) are fulfilled. Then the
statements of Theorem \ref{T1} hold.
\end{theorem}

\noindent
If $a\equiv 0$ then Theorems \ref{T1} and \ref{T2} coincide.
Theorem \ref{T1} is proven in Section \ref{S2}, the proof of
Theorem \ref{T2} is the contents of Section \ref{S4}. Several
remarks are postponed to Section \ref{S3}.

%--------------------------------------------------------------------
\section{Proof of Theorem \ref{T1}}
\label{S2}

Throughout this section, we denote by $c$ a generic constant,
which is always independent of solution $\bfv$. Numbered constants
have the same value (also independent of $\bfv$) in the whole
paper.

Suppose that solution $\bfv$ satisfies condition (a). Then it also
satisfies the assumptions of the so called {\it Theor\`eme de
Structure,} see \cite[p.~57]{Ga2}. (The theorem was in fact for
the first time formulated by Leray in \cite[pp.~244, 245]{Le}.)
Due to this theorem, there exists a system
$\{(a_{\gamma},b_{\gamma})\}_{\gamma\in\Gamma}$ of disjoint open
intervals in $(0,T)$ such that the measure of
$(0,T)\smallsetminus\cup_{\gamma\in\Gamma}(a_{\gamma},b_{\gamma})$
is zero, $\bfv$ is of the class $C^{\infty}$ on
$\R^3\times(a_{\gamma},b_{\gamma})$ for all $\gamma\in\Gamma$, and
$\|A^{1/2}\bfv\|_{2;\, \R^3}$ is locally bounded in each of the
intervals $(a_{\gamma},b_{\gamma})$. If a singularity develops at
the time instant $b_{\gamma}$ then $\|A^{1/2}\bfv(t)\|_{2;\,
\R^3}\to\infty$ for $t\to b_{\gamma}-$. In this case, we call
$b_{\gamma}$ the {\it epoch of irregularity.} In order to prove
that solution $\bfv$ has no singular points in $Q_T$, it is
sufficient to show that there are no epochs of irregularity in
$(0,T)$. Assume, therefore, that $t\in(a_{\gamma},b_{\gamma})$ for
some fixed $\gamma\in\Gamma$.

The Navier--Stokes equation (\ref{1.1}) (with $\nu=1$) can also be
written in the equivalent form
\begin{equation}
\partial_t\bfv+\bfomega\times\bfv+\curl^2\bfv\ =\
-\nabla\bigl(p+{\textstyle\frac{1}{2}}\br|\bfv|^2\bigr).
\label{1.1a}
\end{equation}
Multiplying this equation by $A\bfv$, and integrating in $\R^3$,
we obtain
\begin{equation}
\frac{\rmd}{\rmd\br t}\, \frac{1}{2}\, \|A^{1/2}\bfv\|_{2;\,
\R^3}^2-2\br\bigl(\bfomega^+\times\bfv,\, \bfomega^-\bigr)_{2;\,
\R^3}+\|A^{3/2}\bfv\|_{2;\, \R^3}^2\ =\ 0. \label{2.17}
\end{equation}
We have used the identities
\begin{align*}
[\bfomega\times\bfv]\cdot A\bfv\, &=\, [(\bfomega^++\bfomega^-)
\times\bfv]\cdot(\bfomega^+-\bfomega^-)\, =\,
-[\bfomega^+\times\bfv]\cdot\bfomega^-+[\bfomega^-\times\bfv]\cdot
\bfomega^+ \\ \noalign{\vskip 4pt}
&=\, -2\br[\bfomega^+\times\bfv]\cdot\bfomega^-.
\end{align*}
The scalar product $(\bfomega^+\times\bfv,\, \bfomega^-)_{2;\,
\R^3}$ can be estimated: \phantom{$\cn02$}
\begin{align}
\bigl|\bigl(\bfomega^+ & \times\bfv,\, \bfomega^-\bigr)_{2;\,
\R^3} \bigr|\ \leq\ \|\bfomega^+\|_{3;\, \R^3}\, \|\bfv\|_{3;\,
\R^3}\, \|\bfomega^-\|_{3;\, \R^3}\ \nonumber \\
\noalign{\vskip 2pt}
& \leq\ \cc01^3\, \|A^{1/2}\bfomega^+\|_{2;\, \R^3}\,
\|A^{1/2}\bfv\|_{2;\,
\R^3}\, \|A^{1/2}\bfomega^-\|_{2;\, \R^3} \nonumber \\
\noalign{\vskip 2pt}
& \leq\ \frac{1}{4}\, \bigl\|A^{1/2}\bfomega^-\bigr\|_{2;\, \R^3}
^2+\cc01^6\, \|A^{1/2}\bfv\|_{2;\, \R^3}^2\,
\|A^{1/2}\bfomega^+\|_{2;\, \R^3}^2 \nonumber \\ \noalign{\vskip
2pt}
& \leq\ \frac{1}{4}\, \bigl\|A^{3/2}\bfv\bigr\|_{2;\, \R^3}^2+
\cc01^6\, \|A^{1/2}\bfv\|_{2;\, \R^3}^2\,
\|A^{1/2}\bfomega^+\|_{2;\, \R^3}^2\br. \label{2.1}
\end{align}
Equation (\ref{2.17}) and inequalities (\ref{2.1}) yield
\begin{equation}
\frac{\rmd}{\rmd\br t}\ \|A^{1/2} \bfv\|_{2;\,
\R^3}^2+\|A^{3/2}\bfv\|_{2;\, \R^3}^2\ \leq\ 4\cc01^6\,
\|A^{1/2}\bfv\|_{2;\, \R^3}^2\, \|A^{1/2}\bfomega^+\|_{2;\,
\R^3}^2\br. \label{2.2}
\end{equation}

\noindent
{\bf The case of condition (i).} \ If condition (i) of Theorem
\ref{T1} is fulfilled then the term $\|A^{1/2}\bfomega^+\|_{2;\,
\R^3}^2$ on the right hand side of (\ref{2.2}) is in $L^1(0,T)$.
Hence we can choose $\tau\in(a_{\gamma},b_{\gamma})$ and apply
Gronwall's inequality to (\ref{2.2}) on the time interval
$[\tau,b_{\gamma})$. In this way, we show that
$\|A^{1/2}\bfv\|_{2;\, \R^3}$ is bounded on the interval
$[\tau,b_{\gamma})$, which means that $b_{\gamma}$ is not an epoch
of irregularity.

\vspace{4pt} \noindent
{\bf The case of condition (ii).} \ Let us further assume that
condition (ii) of Theorem \ref{T1} holds. This case is much more
subtle and it is considered in the rest of Section \ref{S2}. The
crucial part of the proof is the estimate of
$\|A^{1/2}\bfomega^+\|_{2;\, \R^3}^2$. The next paragraphs head
towards this aim. We derive an estimate at a fixed time instant
$t$, hence we mostly omit for brevity writing $t$ among the
variables of $\bfomega^+$ and other related functions. Recall that
$t$ is supposed to be in the interval $(a_{\gamma},b_{\gamma})$,
where solution $\bfv$ is smooth. The function value $\bfv(t)$ even
belongs to $\bfW^{2,2}(\R^3)$, as follows from \cite[Theorem
6.1]{Ga2}. Hence $\bfomega(t)\in\Ws$ and, consequently,
$\bfomega^+(t)$ also belongs to $\Ws$.

\vspace{4pt} \noindent
{\bf Sets $K^{mn}_{\xi}$, $C^{mn}$ and the partition of function
$\bfomega^+$.} \ In this paragraph, we define sets
$K^{mn}_{\xi}\subset\R^2$, $C^{mn}\subset\R^3$, and we
successively introduce auxiliary functions $\eta^{mn}$,
$\bfV^{mn}$, $y^{kl}_{mn}$ and $z^{kl}_{mn}$ (for $m,n\in\Z$ and
$k\in\{m-1;\, m;\, m+1\}$, $l\in\{n-1;\, n;\, n+1\}$).

Let us say in advance that $K^{mn}_2$ is a square in $\R^2$ with
the sides of length $5$ and $C^{mn}=K^{mn}_2\times\R$. Using the
functions $\eta^{mn}$ and $\bfV^{mn}$, we create a partition of
function $\bfomega^+$ which consists of functions $\bfomega^{mn}$
such that $\supp\br \bfomega^{mn}\subset C^{mn}$. In following
paragraphs, we derive certain estimates of $\bfomega^{mn}$ (based
on estimates of the auxiliary functions $\eta^{mn}$, $\bfV^{mn}$,
$y^{kl}_{mn}$, $z^{kl}_{mn}$ on sets $C^{mn}$), which strongly use
the structure $C^{mn}=K^{mn}_2\times\R$ of sets $C^{mn}$ and the
fact that $K^{mn}_2$ are squares in $\R^2$ with the length of the
sides independent of $m$, $n$. Then, using the expansion
$\bfomega^+=\sum_{m,n\in\Z}\br\bfomega^{mn}$, we derive an
estimate of $A^{1/2}\bfomega^+$ which is needed in (\ref{2.2}).
(See estimate (\ref{2.14}).)

We begin with the definition of sets $K^{mn}_{\xi}\subset\R^2$ and
$C^{mn}\subset\R^3$: \ for $m,n\in\Z$ and
$\xi\in(-\frac{1}{2},\infty)$, we denote $K^{mn}_{\xi}:=
(m-\xi,m+1+\xi)\times (n-\xi,n+1+\xi)$. Further, we put
$C^{mn}:=K^{mn}_2\times\R=(m-2,m+3)\times(n-2,n+3)\times\R\subset\R^3$.
Thus, $K^{mn}_{\xi}$ are squares in $\R^2$, while $C^{mn}$ are
cylinders in $\R^3$.

Let $\epsilon\in(0,\frac{1}{8})$ be fixed. There exists a
partition of unity with these properties: the partition consists
of the system $\{\eta^{mn}\}_{m,n\in\Z}$ of infinitely
differentiable functions of two variables, such that

\vspace{-8pt}
\begin{list}{}
{\setlength{\topsep 0pt}
\setlength{\itemsep 1pt}
\setlength{\leftmargin 45pt}
\setlength{\rightmargin 0pt}
\setlength{\labelwidth 6pt}}

\item[a) ]
$\eta^{mn}=1\, $ in $K^{mn}_{-\epsilon}$, \ \ $\eta^{mn}=0\, $ in
$\R^2\smallsetminus K^{mn}_{\epsilon}$, \ \ $0\leq\eta^{mn}\leq
1\, $ in $\R^2$,

\item[b) ]
$\eta^{m+i,n+j}(x_1,x_2)=\eta^{mn}(x_1+i,x_2+j)\, $ for all $i,\br
j\in\Z$,

\item[c) ]
$\sum_{m,n\in\Z}\, \eta^{mn}=1\, $ in $\R^2$.

\end{list}

\noindent
(Function $\eta^{mn}$ can be e.g.~defined by means of a mollifier
with the kernel supported on $B_{\epsilon}(\bfzero)$, applied to
the characteristic function of the square $K^{mn}_0$.)

We denote by $\nabla_{2D}^{}$ the 2D nabla operator
$(\partial_1,\partial_2)$, and by $\bfomega^+_{2D}$ the 2D vector
field $(\omega^+_1,\omega^+_2)$. Applying successively the
procedure of solving the equation $\nabla_{2D}\cdot\bfu=f$,
especially the so called Bogovskij formula (see e.g.~\cite{BoSo}),
we deduce that there exists a system $\{\bfV^{mn}\}_{m,n\in\Z}$ of
2D vector functions $\bfV^{mn}=(V^{mn}_1,V^{mn}_2)$ defined in
$\R^3$ with the properties

\begin{list}{}
{\setlength{\topsep 2pt}
\setlength{\itemsep 1pt}
\setlength{\leftmargin 45pt}
\setlength{\rightmargin 0pt}
\setlength{\labelwidth 6pt}}

\item[d) ]
$\nabla_{2D}\cdot\bfV^{mn}=-\nabla_{2D}^{}\eta^{mn}
\cdot\bfomega^+_{2D}\, $ in $\R^3$,

\item[e) ]
$\supp\, \bfV^{mn}\subset [K^{mn}_{2\epsilon}\smallsetminus
K^{mn}_{-2\epsilon}]\times\R$,

\item[f) ]
$\sum_{m,n\in\Z}\br\bfV^{mn}=\bfzero\, $ in $\R^3$,

\item[g) ]
$\|\bfV^{mn}\|_{2;\, C^{mn}}+\|\nabla_{2D}\bfV^{mn}\|_{2;\,
C^{mn}}\leq c\, \|\bfomega^+_{2D}\|_{2;\, C^{mn}}$,

\item[h) ]
$\|\partial_3\bfV^{mn}\|_{2;\, C^{mn}}\leq c\,
\|\partial_3\bfomega^+_{2D}\|_{2;\, C^{mn}}$.

\end{list}

\noindent
Constant $c$ is always independent of $m$ and $n$. We can derive
from the last two estimates, by interpolation, that
\vspace{-8pt}
\begin{equation}
\|\bfV^{mn}\|_{1/2,2;\, C^{mn}}\ \leq\ c\,
\|\bfomega^+\|_{1/2,2;\, C^{mn}}\br. \label{2.4}
\end{equation}
For technical reasons, we put $V^{mn}_3:=0$ and we further
consider $\bfV^{mn}$ to be the 3D vector field. Further, we put
\vspace{-8pt}
\begin{displaymath}
\bfomega^{mn}\ :=\ \eta^{mn}\br\bfomega^+-\bfV^{mn}.
\end{displaymath}
The components of $\bfomega^{mn}$ are denoted by $\omega^{mn}_1$,
$\omega^{mn}_2$ and $\omega^{mn}_3$. By analogy with
$\bfomega^+_{2D}$, we also denote
$\bfomega^{mn}_{2D}:=(\omega^{mn}_1,\omega^{mn}_2)$. Function
$\bfomega^{mn}$ is divergence--free in $\R^3$, it equals
$\bfomega^+$ in $K^{mn}_{-2\epsilon}\times\R$, and its support is
a subset of $K^{mn}_{2\epsilon}\times\R$. Moreover, we have
$\bfomega^+=\sum_{m,n\in\Z}\br\bfomega^{mn}$.

The term $\|A^{1/2}\bfomega^+\|_{2;\, \R^3}^2$ can now be written
in this form:
\begin{align}
\|A^{1/2}\bfomega^+\|_{2;\, \R^3}^2\, & =\,
(A\bfomega^+,\bfomega^+)_{2;\, \R^3}\, =\, (\curl\,
\bfomega^+,\bfomega^+)_{2;\, \R^3}\, =
\sum_{m,n\in\Z}\br\sum_{k,\br l\in\Z}\, (\curl\,
\bfomega^{mn},\bfomega^{kl})_{2;\, \R^3}\  \nonumber \\
\noalign{\vskip 0pt}
& =\ \sum_{m,n\in\Z}\ \ \sum_{\substack{k\in\{m-1;\, m;\, m+1\} \\
l\in\{n-1;\, n;\, n+1\}}}\, (\curl\,
\bfomega^{mn},\bfomega^{kl})_{2;\, C^{mn}}\br. \label{2.3}
\end{align}
The last equality holds because the supports of $\bfomega^{mn}$
and $\bfomega^{kl}$ have a non--empty intersection only if
$k\in\{m-1;\, m;\, m+1\}$ and $l\in\{n-1;\, n;\, n+1\}$. In this
case, both the supports are subsets of $C^{mn}$.

\vspace{4pt} \noindent
{\bf Operator $(-\Delta)_{mn}$.} \ We denote by $(-\Delta)_{mn}$
the operator $-\Delta$ with the domain
$D((-\Delta)_{mn}):=W^{2,2}(C^{mn})\cap W^{1,2}_0(C^{mn})$.
Operator $(-\Delta)_{mn}$ is positive and self--adjoint in
$L^2(C^{mn})$, with a bounded inverse. The powers of
$(-\Delta)_{mn}$, with positive as well as negative exponents, can
be defined in the usual way by means of the corresponding spectral
expansion, see e.g.~\cite{Ka}.

\vspace{4pt} \noindent
{\bf Auxiliary functions $y^{kl}_{mn}$.} \ We denote by
$y^{kl}_{mn}$ the solution of the 2D Neumann problem
\begin{equation}
\Delta_{2D}^{}y^{kl}_{mn}=
-(-\Delta)_{mn}^{1/4}\br(\partial_3\omega^{kl}_3)\ \ \mbox{in}\
K^{mn}_2, \qquad \frac{\partial y^{kl}_{mn}}{\partial\bfn}=0\ \
\mbox{on}\ \partial K^{mn}_2 \label{2.7}
\end{equation}
for $m,\, n\in\Z$, $k\in\{m-1;\, m;\, m+1\}$ and $l\in\{n-1;\,
n;\, n+1\}$. Function $y^{kl}_{mn}$ satisfies the estimate
\vspace{-8pt}
\begin{equation}
\|\nabla_{2D}^{}y^{kl}_{mn}\|_{2;\,
K^{mn}_2}^2+\|\nabla_{2D}^2y^{kl}_{mn}\|_{2;\, K^{mn}_2}^2\ \leq\
c\, \|(-\Delta)_{mn}^{1/4}\br(\partial_3\omega^{kl}_3)\|_{2;\,
K^{mn}_2}^2, \label{2.16}
\end{equation}
where $c$ is independent of $m,\, n,\, k$ and $l$. Since
$\partial_3\omega^{kl}_3$ is a function of three variables $x_1,\,
x_2,\, x_3$, function $y^{kl}_{mn}$ naturally depends not only on
$x_1,\, x_2$, but also on $x_3$. Integrating the last estimate
with respect to $x_3$, we obtain
\begin{equation}
\|\nabla_{2D}^2y^{kl}_{mn}\|_{2;\,
C^{mn}}^2+\|\nabla_{2D}^{}y^{kl}_{mn}\|_{2;\, C^{mn}}^2\ \leq\ c\,
\|(-\Delta)_{mn}^{1/4}\br\partial_3\omega^{kl}_3\|_{2;\,
C^{mn}}^2\br. \label{2.13}
\end{equation}

\noindent
{\bf Auxiliary functions $z^{kl}_{mn}$.} \ We define function
$z^{kl}_{mn}$ to be the solution of the equation
\begin{equation}
\nabla_{2D}^{\perp}\br z^{kl}_{mn}\ =\
(-\Delta)_{mn}^{1/4}\bfomega^{kl}_{2D}-\nabla_{2D}^{} y^{kl}_{mn}
\label{2.15}
\end{equation}
in $K^{mn}_2$. (Here, we denote by $\nabla^{\perp}_{2D}$ the
operator $(-\partial_2,\partial_1)$.) The solution exists because
\begin{displaymath}
\nabla_{2D}^{}\cdot\bigl[(-\Delta)_{mn}^{1/4}\br\bfomega^{kl}_{2D}-
\nabla_{2D}^{}y^{kl}_{mn}\bigr]\ =\ 0.
\end{displaymath}
Solution $z^{kl}_{mn}$ depends not only on $x_1$, $x_2$, but also
on $x_3$ because the right hand side of equation (\ref{2.15})
depends on $x_3$ as well. Function $z^{kl}_{mn}$ is the so called
{\it stream function} of the 2D vector field
$(-\Delta)_{mn}^{1/4}\br\bfomega^{kl}_{2D}-
\nabla_{2D}^{}y^{kl}_{mn}$. For each fixed $x_3\in\R$,
$z^{kl}_{mn}$ satisfies the estimate
\begin{equation}
\|\nabla_{2D}^{}z^{kl}_{mn}\|_{2;\, K^{mn}_2}\ \leq\ c\, \bigl(
\|(-\Delta)_{mn}^{1/4}\br\bfomega^{kl}_{2D}\|_{2;\,
K^{mn}_2}+\|\nabla_{2D}^{} y^{kl}_{mn}\|_{2;\, K^{mn}_2}\bigr).
\label{2.8}
\end{equation}
Moreover, $z^{kl}_{mn}$ is constant on $\partial C^{mn}$
($=\partial K^{mn}_2\times\R$). This follows from the identities
\begin{displaymath}
\nabla_{2D}^{\perp}z^{kl}_{mn}\cdot\bfn\ =\
(-\Delta)_{mn}^{1/4}\bfomega^{kl}_{2D}\cdot\bfn-
\nabla_{2D}^{}y^{kl}_{mn}\cdot\bfn\ =\ 0,
\end{displaymath}
valid on $\partial C^{mn}$. Indeed, the second term
$\nabla_{2D}^{}y^{kl}_{mn}\cdot\bfn$ equals zero on $\partial
C^{mn}$ by definition of $y^{kl}_{mn}$. The first term
$(-\Delta)_{mn}^{1/4}\br\bfomega^{kl}$ is zero on $\partial
C^{mn}$ because $\bfomega^{mn}\in D((-\Delta)_{mn})$, hence
$(-\Delta)_{mn}^{1/4}\bfomega^{mn}\in D((-\Delta)_{mn}^{3/4})$,
and functions from $D((-\Delta)_{mn}^{3/4})$ have the trace on
$\partial C^{mn}$ equal to zero. (This can be easily verified
because $D((-\Delta)_{mn}^{1/2})=W^{1,2}_0(C^{mn})$, which implies
that $D((-\Delta)_{mn}^{3/4})$ is the interpolation space between
$D((-\Delta)_{mn})\equiv W^{2,2}(C^{mn})\cap W^{1,2}_0(C^{mn})$
and $W^{1,2}_0(C^{mn})$, and both the spaces contain only
functions whose traces are equal to zero on $\partial C^{mn}$.)
Function $z^{kl}_{mn}$ is unique up to an additive function of $t$
and $x_3$. We can now choose this function so that $z^{kl}_{mn}=0$
on $\partial C^{mn}$. This choice, together with (\ref{2.8}) and
(\ref{2.13}), implies that
\begin{align}
\|z^{kl}_{mn}\|_{2;\, C^{mn}}\ &\leq\ c\, \bigl(
\|(-\Delta)_{mn}^{1/4}\bfomega^{kl}_{2D}\|_{2;\,
C^{mn}}+\|\nabla_{2D}^{} y^{kl}_{mn}\|_{2;\, C^{mn}}\bigr)
\nonumber \\ \noalign{\vskip 2pt}
& \leq\ c\, \bigl( \|(-\Delta)_{mn}^{1/4}\bfomega^{kl}\|_{2;\,
C^{mn}}+ \|(-\Delta)_{mn}^{1/4}\br(\partial_3
\omega^{kl}_3)\|_{2;\, C^{mn}}\bigr). \label{2.5}
\end{align}

\noindent
{\bf The estimate of $(\curl\, \bfomega^{mn},\bfomega^{kl})_{2;\,
C^{mn}}$.} \ We denote
$\bfw^{mn}\equiv(w^{mn}_1,w^{mn}_2,w^{mn}_3):=\curl\,
\bfomega^{mn}$ and $\bfw^{mn}_{2D}:=(w^{mn}_1,w^{mn}_2)$. We
always assume that $k\in\{m-1;\, m;\, m+1\}$ and $l\in\{n-1;\,
n;\, n+1\}$. Due to the definition of functions $y^{kl}_{mn}$ and
$z^{kl}_{mn}$, function $(-\Delta)_{mn}^{1/4}\bfomega^{kl}$ has
the form
\begin{displaymath}
(-\Delta)_{mn}^{1/4}\bfomega^{kl}\ =\ \left( \begin{array}{c}
\partial_1y^{kl}_{mn} \\ [2pt] \partial_2y^{kl}_{mn} \\ [2pt]
(-\Delta)_{mn}^{1/4}\br\omega^{kl}_3 \end{array}
\right)+\curl\left( \begin{array}{c} 0 \\ 0 \\ z^{kl}_{mn}
\end{array} \right) \qquad \mbox{in}\ C^{mn}.
\end{displaymath}
Hence
\begin{align}
& (\curl\, \bfomega^{mn},\bfomega^{kl})_{2;\, C^{mn}}\ =\
(\bfw^{mn},\bfomega^{kl})_{2;\, C^{mn}}\ =\ \int_{C^{mn}}
(-\Delta)_{mn}^{-1/4}\br\bfw^{mn}\cdot(-\Delta)_{mn}^{1/4}\br
\bfomega^{kl}\; \rmd\bfx \nonumber \\
\noalign{\vskip 0pt}
& =\ \int_{C^{mn}} \left[(-\Delta)_{mn}^{-1/4}\br\bfw^{mn}\cdot
\left(
\begin{array}{c} \partial_1y^{kl}_{mn} \\ [2pt] \partial_2y^{kl}_{mn} \\
[2pt] (-\Delta)_{mn}^{1/4}\br\omega^{kl}_3 \end{array}
\right)+(-\Delta)_{mn}^{-1/4}\curl^2\bfomega^{mn}\cdot
\left(\begin{array}{c}
0 \\ 0 \\ z^{kl}_{mn} \end{array} \right)\right]\; \rmd\bfx \nonumber \\
\noalign{\vskip 2pt}
& =\ \int_{C^{mn}} \left[(-\Delta)_{mn}^{-1/4}\br\bfw^{mn}\cdot
\left(
\begin{array}{c} \partial_1y^{kl}_{mn} \\ [2pt]
\partial_2y^{kl}_{mn} \\ [2pt] (-\Delta)_{mn}^{1/4}\br\omega^{kl}_{3}
\end{array} \right)+(-\Delta)_{mn}^{3/4}\br\omega^{mn}_3\,
z^{kl}_{mn}\right]\; \rmd\bfx \nonumber \\ \noalign{\vskip 4pt}
& =\ \int_{C^{mn}}\bigl\{ (-\Delta)_{mn}^{-1/4}\bfw^{mn}_{2D}\cdot
\nabla_{2D}y^{kl}_{mn}+(-\Delta)_{mn}^{-1/4}\br w^{mn}_3\,
(-\Delta)_{mn}^{1/4}\br \omega^{kl}_3+
(-\Delta)_{mn}^{3/4}\omega^{mn}_3\,
z^{kl}_{mn}\bigr\}\; \rmd\bfx \nonumber \\
\noalign{\vskip 2pt}
& \leq\ c\, \|\bfomega^{mn}\|_{1/2,2;\, C^{mn}}\,
\|\nabla_{2D}^{}y^{kl}_{mn}\|_{2;\, C^{mn}}+c\,
\|\bfomega^{mn}\|_{1/2,2;\, C^{mn}}\, \|\omega^{kl}_3\|_{1/2,2;\,
C^{mn}} \nonumber \\ \noalign{\vskip 4pt}
& \hspace{40pt} +\|\omega^{mn}_3\bigr\|_{3/4,2;\, C^{mn}}\,
\|z^{kl}_{mn}\|_{2;\, C^{mn}}\, \nonumber
\\ \noalign{\vskip 4pt}
& \leq\ c\, \|\bfomega^{mn}\|_{1/2,2;\, C^{mn}}\,
\|(-\Delta)_{mn}^{1/4}\br(\partial_3\omega^{kl}_3)\|_{2;\,
C^{mn}}+c\, \|\bfomega^{mn}\|_{1/2,2;\, C^{mn}}\,
\|\omega^{kl}_3\|_{1/2,2;\, C^{mn}} \nonumber \\ \noalign{\vskip
4pt}
& \hspace{40pt} +c\, \|\omega^{mn}_3\bigr\|_{3/2,2;\, C^{mn}}\,
\bigl( \|(-\Delta)_{mn}^{1/4}\bfomega^{kl}\|_{2;\, C^{mn}}+
\|(-\Delta)_{mn}^{1/4}\br(\partial_3 \omega^{kl}_3)\|_{2;\,
C^{mn}}\bigr). \label{2.12a}
\end{align}
Each term on the right hand side contains some norm of
$\omega^{mn}_3$ ($=\eta^{mn}\omega^+_3$) or $\omega^{kl}_3$
($=\eta^{kl}\omega^+_3$). This is how the third component
$\omega^+_3$ controls the scalar product $(\curl\,
\bfomega^{mn},\bfomega^{kl})_{2;\, C^{mn}}$. The right hand side
of (\ref{2.12a}) is further less than or equal to
\begin{align}
c\, & \|\bfomega^{mn}\|_{1/2,2;\, C^{mn}}\,
\|\omega^{kl}_3\|_{3/2,2;\, C^{mn}}+ c\,
\|\omega^{mn}_3\|_{3/2,2;\, C^{mn}}\, \|\bfomega^{kl}\|_{3/2,2;\,
C^{mn}} \nonumber \\ \noalign{\vskip 4pt}
& \hspace{40pt} +c\, \|\omega^{mn}_3\|_{3/2,2;\, C^{mn}}\,
\|\omega^{kl}_3\|_{3/2,2;\, C^{mn}} \nonumber \\
\noalign{\vskip 4pt}
& \leq\ \delta\, \|\bfomega^{mn}\|_{1/2,2;\, C^{mn}}^2+\delta\,
\|\bfomega^{kl}\|_{1/2,2;\, C^{mn}}^2+ c(\delta)\,
\|\omega^{mn}_3\|_{3/2,2;\, C^{mn}}^2+c(\delta)\,
\|\omega^{kl}_3\|_{3/2,2;\, C^{mn}}^2 \nonumber \\ \noalign{\vskip
4pt}
& \leq\ \delta\, \|\eta^{mn}\bfomega^+\|_{1/2,2;\,
C^{mn}}^2+\delta\, \|\bfV^{mn}\|_{1/2,2;\, C^{mn}}^2+\delta\,
\|\eta^{kl}\bfomega^+\|_{1/2,2;\, C^{mn}}^2+\delta\,
\|\bfV^{kl}\|_{1/2,2;\, C^{mn}}^2 \nonumber \\ \noalign{\vskip
4pt}
& \hspace{40pt} + c(\delta)\, \|\eta^{mn}\omega^+_3\|_{3/2,2;\,
C^{mn}}^2+c(\delta)\, \|\eta^{kl}\omega^+_3\|_{3/2,2;\,
C^{mn}}^2\br. \label{2.12}
\end{align}
The norm $\|\bfV^{mn}\|_{1/2,2;\, C^{mn}}$ can be estimated by
means of (\ref{2.4}). Since $\bfV^{kl}$ is supported inside
$C^{mn}$, one can also derive (by analogy with (\ref{2.4})) that
$\|\bfV^{kl}\|_{1/2,2;\, C^{mn}}\leq c\, \|\bfomega^+\|_{1/2,2;\,
C^{mn}}$. Furthermore the norm $\|\eta^{mn}\bfomega^+\|_{1/2,2;\,
C^{mn}}$ can be estimated by $c\, \|\bfomega^+\|_{1/2,2;\,
C^{mn}}$. (This can be easily proven in the same way as Theorem
I.7.3 in \cite{LiMa}.) The other terms on the right hand side of
(\ref{2.12}) that contain functions $\eta^{mn}$ or $\eta^{kl}$ can
be estimated similarly. Thus, (\ref{2.12}) yields
\begin{equation}
(\curl\, \bfomega^{mn},\bfomega^{kl})_{2;\, C^{mn}}\ \leq\
\delta\, c\, \|\bfomega^+\|_{1/2,2;\, C^{mn}}^2+c(\delta)\,
\|\omega^+_3\|_{3/2,2;\, C^{mn}}^2\br. \label{2.10}
\end{equation}

\noindent
{\bf The estimate of the right hand side of (\ref{2.3}).} \ The
sum $\sum_{m,n\in\Z}$ in (\ref{2.3}) can be split to twenty five
parts, which successively contain the sums over $m=0$ mod $5$,
$\dots,$ $m=4$ mod $5$ and $n=0$ mod $5$, $\dots,$ $n=4$ mod $5$.

Let us consider e.g.~the case $m,n\in\Z$, $m=0$ mod $5$, $n=0$ mod
$5$ (i.e.~$m$ and $n$ are integer multiples of $5$). Denote the
sum over these $m,\, n$ by $\sum_{m,n\in\Z}^{(1)}$, and the sums
over twenty four other possibilities by $\sum_{m,n\in\Z}^{(2)}$,
$\dots,$ $\sum_{m,n\in\Z}^{(25)}$. The cylinders $C^{mn}$
corresponding to the first case are disjoint and their union
equals $\R^3$ up to the set of measure zero. Applying
(\ref{2.10}), we have
\vspace{-4pt}
\begin{align}
{\sum_{m,n\in\Z}}^{\hspace{-4pt} (1)}\ &
\sum_{\substack{k\in\{m-1;\, m;\, m+1\} \\ l\in\{n-1;\, n;\,
n+1\}}}\, (\curl\, \bfomega^{mn},\bfomega^{kl})_{2;\, C^{mn}}
\nonumber \\ \noalign{\vskip 0pt}
& \leq\ \delta\, c\, {\sum_{m,n\in\Z}}^{\hspace{-4pt} (1)}
\|\bfomega^+\|_{1/2,2;\, C^{mn}}^2+c(\delta)
{\sum_{m,n\in\Z}}^{\hspace{-4pt} (1)}\|\omega^+_3\|_{3/2,2;\,
C^{mn}}^2\br. \label{2.11}
\end{align}

\vspace{-8pt} \noindent
Obviously, the $L^2$--norms and $W^{1,2}$--norms of $\bfomega^+$
satisfy the identities
\vspace{-4pt}
\begin{displaymath}
{\sum_{m,n\in\Z}}^{\hspace{-4pt} (1)} \|\bfomega^+\|_{2;\,
C^{mn}}^2\ =\ \|\bfomega^+\|_{2;\,\R^3}^2 \quad \mbox{and} \quad
{\sum_{m,n\in\Z}}^{\hspace{-4pt} (1)} \|\bfomega^+\|_{1,2;\,
C^{mn}}^2\ =\ \|\bfomega^+\|_{1,2;\,\R^3}^2\br.
\end{displaymath}

\vspace{-4pt} \noindent
Applying appropriately the theorem on interpolation (see
\cite[Theorem I.5.1]{LiMa}), we derive that
\vspace{-4pt}
\begin{displaymath}
{\sum_{m,n\in\Z}}^{\hspace{-4pt} (1)}\br\|\bfomega^+\|_{1/2,2;\,
C^{mn}}^2\ \leq\ c\, \|\bfomega^+\|_{1/2,2;\, \R^3}^2\br.
\end{displaymath}

\vspace{-4pt} \noindent
The norms $\|\omega^+_3\|_{3/2,2;\, C^{mn}}$ and
$\|\omega^+_3\|_{3/2,2;\, \R^3}$ satisfy the same inequalities.
Applying these inequalities, and estimating the sums
$\sum_{m,n\in\Z}^{(2)}$, $\dots,$ $\sum_{m,n\in\Z}^{(25)}$ in the
same way as the sum in (\ref{2.11}), we get
\vspace{-8pt}
\begin{align*}
\|A^{1/2}\bfomega^+\|_{2;\, \R^3}^2\ & \leq\
\sum_{m,n\in\Z}\ \ \sum_{\substack{k\in\{m-1;\, m;\, m+1\} \\
l\in\{n-1;\, n;\, n+1\}}}\, (\curl\,
\bfomega^{mn},\bfomega^{kl})_{2;\, C^{mn}} \nonumber \\
\noalign{\vskip 0pt}
& \leq\ \delta\, c\, \|\bfomega^+\|_{1/2,2;\, \R^3}^2+c(\delta)\,
\|\omega^+_3\|_{3/2,2;\, \R^3}^2\br.
\end{align*}
The first term on the right hand side is less than or equal to
$\delta\, c\, \bigl(\|\bfomega^+\|_{2;\,
\R^3}^2+\|A^{1/2}\bfomega^+\|_{2;\, \R^3}^2\bigr)$. Choosing
$\delta>0$ so small that $\delta\, c\leq\frac{1}{2}$, and
estimating $\|\omega^+_3\|_{3/2,2;\, \R^3}^2$ from above by
$\|\omega^+_3\|_{2;\, \R^3}^2+\|(-\Delta)^{3/4}\omega^+_3\|_{2;\,
\R^3}^2$, we finally obtain \phantom{$\cn05\cn06$}
\begin{equation}
\|A^{1/2}\bfomega^+\|_{2;\, \R^3}^2\ \leq\ \cc05\,
\|\bfomega^+\|_{2 ;\, \R^3}^2+\cc06\,
\|(-\Delta)^{3/4}\omega^+_3\|_{2;\, \R^3}^2\br.
\label{2.14}
\end{equation}

\noindent
{\bf Completion of the proof.} \ Substituting estimate
(\ref{2.14}) to (\ref{2.2}), we get
\begin{align}
\frac{\rmd}{\rmd\br t}\ & \frac{1}{2}\ \|A^{1/2}\bfv\|_{2;\,
\R^3}^2+\|A^{3/2}\bfv\|_{2;\, \R^3}^2 \nonumber \\
\noalign{\vskip 4pt}
&\leq\ 4\cc01^6\, \|A^{1/2}\bfv\|_{2;\, \R^3}^2\, \bigl( \cc05\,
\|\bfomega^+\|_{2 ;\, \R^3}^2+\cc06\,
\|(-\Delta)^{3/4}\omega^+_3\|_{2;\, \R^3}^2\bigr).
\label{2.19}
\end{align}
Recall that this inequality holds for $t\in
(a_{\gamma},b_{\gamma})$. The expression in parentheses on the
right hand side is integrable as a function of $t$ in $(0,T)$.
Thus, we can again choose $\tau\in(a_{\gamma},b_{\gamma})$ and
apply Gronwall's inequality to (\ref{2.19}) on the interval
$[\tau,b_{\gamma})$. This is how we show that
$\|A^{1/2}\bfv\|_{2;\, \R^3}$ is bounded on $[\tau,b_{\gamma})$.
Consequently, $b_{\gamma}$ cannot be the epoch of irregularity of
solution $\bfv$ and $\|A^{1/2}\bfv\|_{2;\, \R^3}$ is therefore
bounded on $[\tau,T)$. Since $\tau$ can be chosen arbitrarily
close to $0$, we have proven that $\|A^{1/2}\bfv\|_{2;\, \R^3}$ is
bounded on each interval of the type $(\vartheta,T)$ for
$0<\vartheta<T$.

If solution $\bfv$ satisfies condition (b) of Theorem \ref{T1}
then the initial velocity $\bfv_0$ belongs to the space
$D(A^{1/2})$. Hence there exists $T^*\in(0,T]$ and a strong
solution $\bfv^*$ of the problem (\ref{1.1})--(\ref{1.3}), whose
norm $\|A^{1/2}\bfv^*\|_{2;\, \R^3}$ is locally bounded on
$[0,T^*)$. (See e.g.~\cite[Section V.4]{So}.) The considered weak
solution $\bfv$ coincides with $\bfv^*$ on $(0,T^*)$ by the
theorem on uniqueness, see \cite[Theorem 4.2]{Ga2}. (This is the
point where we use the fact that $\bfv$ satisfies (EI).) The time
instant $T^*$ is either an epoch of irregularity (if
$\|A^{1/2}\bfv(t)\|_{2;\, \R^3}\to\infty$ for $t\to T^*-$) or
$T*=T$ and $\|A^{1/2}\bfv\|_{2;\, \R^3}$ is bounded on $(0,T)$.
Repeating the procedure from the previous paragraphs, we can show
that $T^*$ cannot be the epoch of irregularity. Thus,
$\|A^{1/2}\bfv\|_{2;\, \R^3}$ is bounded on $(0,T)$ and solution
$\bfv$ has therefore no singular points in $Q_T$.

%----------------------------------------------------------------------
\section{Proof of Theorem \ref{T2}}
\label{S4}

We can at first copy the proof of Theorem \ref{T1} in Section
\ref{S2} up to inequality (\ref{2.2}). Instead of ``the case of
condition (i)'', we consider ``the case of condition (iii)''.
Recall that $F_{\lambda}$ (the resolution of identity associated
with operator $A$) is, for $\lambda\geq 0$, related to
$E_{\lambda}$ (the resolution of identity associated with operator
$\curl$) by the formula $F_{\lambda}=E_{\lambda}-E_{-\lambda}$.
Thus, for $t\in(a_{\gamma},b_{\gamma})$ we have
\begin{align}
\|A^{1/2} & \bfomega^+(t)\|_{2;\, \R^3}^2\ =\
\bigl(A\bfomega^+(t),\br\bfomega^+(t)\bigr)_{2;\, \R^3}\ =\
\int_0^{\infty}\lambda\; \rmd\bigl(F_{\lambda}\bfomega^+(t),\br
\bfomega^+(t)\bigr)_{2;\, \R^3} \nonumber \\ \noalign{\vskip 0pt}
&=\ \int_0^{\infty}\lambda\;
\rmd\bigl((E_{\lambda}-E_{-\lambda})\bfomega^+(t),\br
\bfomega^+(t)\bigr)_{2;\, \R^3}\ =\ \int_0^{\infty}\lambda\;
\rmd\bigl(E_{\lambda}\bfomega^+(t),\br \bfomega^+(t)\bigr)_{2;\,
\R^3} \nonumber \\ \noalign{\vskip 0pt}
& =\ \int_0^{a_+(t)}\lambda\; \rmd\bigl(E_{\lambda}A\bfv^+(t),\br
A\bfv^+(t)\bigr)_{2;\, \R^3}+\int_{a_+(t)}^{\infty}\lambda\;
\rmd\bigl(E_{\lambda}\bfomega^+(t),\br\bfomega^+(t)\bigr)_{2;\,
\R^3}. \label{4.1}
\end{align}
(We have used the identity $E_{-\lambda}\bfomega^+(t)=\bfzero$ for
$\lambda\geq 0$.) As in Section \ref{S2}, we further omit writing
$(t)$. The first integral on the right hand side of (\ref{4.1})
equals \phantom{$\cn15$}
\vspace{0pt}
\begin{equation}
\int_0^{a_+}\lambda^3\; \rmd\bigl(E_{\lambda}\bfv^+,\br
\bfv^+\bigr)_{2;\, \R^3}\ \leq\ a_+^3\int_0^{a_+}
\rmd\bigl(E_{\lambda}\bfv^+,\br\bfv^+\bigr)_{2;\, \R^3}\ \leq
a_+^3\, \|\bfv^+\|_{2;\, \R^3}^2\ \leq\ \cc15\, a_+^3,
\label{4.2}
\end{equation}
where $\cc15$ is the essential upper bound of $\|\bfv\|_{2;\,
\R^3}^2$ on $(0,T)$.

Let us now deal with the second integral on the right hand side of
(\ref{4.1}). If $a\geq 0$ then $\bfomega^+$ can be expressed as
the sum $\bfomega_{(0,a)}+\bfomega_a^+$, where
$\bfomega_{(0,a)}:=\curl\, \bfv_{(0,a)}$ and $\bfv_{(0,a)}:=
\int_0^a\rmd E_{\lambda}\bfv=(E_a-E_0)\bfv$. Thus,
$E_{\lambda}\bfomega^+$ (for $\lambda\geq a=a_+$) equals
$E_{\lambda}\bfomega_{(0,a)}+E_{\lambda}\bfomega_a^+=
\bfomega_{(0,a)}+E_{\lambda}\bfomega_a^+$. The differential of
$(E_{\lambda}\bfomega^+,\bfomega^+)_{2;\, \R^3}$ with respect to
variable $\lambda$ is
\begin{align*}
\rmd(E_{\lambda}\bfomega^+,\bfomega^+)_{2;\, \R^3}\ &=\
\rmd(\bfomega_{(0,a)},\bfomega^+)_{2;\,
\R^3}+\rmd(E_{\lambda}\bfomega_a^+,\bfomega^+)_{2;\, \R^3}\ =\
\rmd(E_{\lambda}\bfomega_a^+,\bfomega^+)_{2;\, \R^3} \\
\noalign{\vskip 2pt}
&=\ \rmd(E_{\lambda}\bfomega_a^+,\bfomega_a^+)_{2;\, \R^3}+
\rmd(E_{\lambda}\bfomega_a^+,\bfomega_{(0,a)})_{2;\, \R^3}\ =\
\rmd(E_{\lambda}\bfomega_a^+,\bfomega_a^+)_{2;\, \R^3}.
\end{align*}
(The last equality holds because $E_{\lambda}\bfomega^+_a$ and
$\bfomega_{(0,a)}$ are orthogonal in $\bfL^2(\R^3)$.) Hence
\begin{align}
\int_{a_+}^{\infty}\lambda\; &\rmd\bigl(E_{\lambda}\bfomega^+,\br
\bfomega^+\bigr)_{2;\, \R^3}\ =\ \int_{a_+}^{\infty}\lambda\;
\rmd\bigl(E_{\lambda}\bfomega_a^+,\br \bfomega_a^+\bigr)_{2;\,
\R^3}\ =\ \int_{a_+}^{\infty}\lambda\;
\rmd\bigl(F_{\lambda}\bfomega_a^+,\br \bfomega_a^+\bigr)_{2;\,
\R^3} \nonumber \\
&=\ \|A^{1/2}\bfomega_a^+\|_{2;\, \R^3}^2. \label{4.3}
\end{align}
Similarly, if $a<0$ then
$\bfomega_a^+=\bfomega_{(a,0)}+\bfomega^+$, where
$\bfomega_{(a,0)}:=\curl\, \bfv_{(a,0)}$ and $\bfv_{(a,0)}:=
\int_a^0\rmd E_{\lambda}\bfv=(E_0-E_a)\bfv$. For $\lambda\geq 0$,
we have $E_{\lambda}\bfomega_a^+=E_{\lambda}\bfomega_{(a,0)}+
E_{\lambda}\bfomega^+=\bfomega_{(a,0)}+E_{\lambda}\bfomega^+$ and
\begin{align*}
\rmd(E_{\lambda}\bfomega^+_a,\bfomega^+_a)_{2;\, \R^3}\ &=\
\rmd(\bfomega_{(a,0)},\bfomega^+_a)_{2;\,
\R^3}+\rmd(E_{\lambda}\bfomega^+,\bfomega^+_a)_{2;\, \R^3}\ =\
\rmd(E_{\lambda}\bfomega^+,\bfomega^+_a)_{2;\, \R^3} \\
\noalign{\vskip 2pt}
&=\ \rmd(E_{\lambda}\bfomega^+,\bfomega_{(a,0)})_{2;\, \R^3}+
\rmd(E_{\lambda}\bfomega^+,\bfomega^+)_{2;\, \R^3}\ =\
\rmd(E_{\lambda}\bfomega^+,\bfomega^+)_{2;\, \R^3}.
\end{align*}
Hence
\begin{align}
\int_{a_+}^{\infty}\lambda\; &\rmd\bigl(E_{\lambda}\bfomega^+,\br
\bfomega^+\bigr)_{2;\, \R^3}\ =\ \int_{0}^{\infty}\lambda\;
\rmd\bigl(E_{\lambda}\bfomega_a^+,\br \bfomega_a^+\bigr)_{2;\,
\R^3} \nonumber \\
&=\ \int_{0}^{\infty}\lambda\;
\rmd\bigl(F_{\lambda}\bfomega_a^+,\br \bfomega_a^+\bigr)_{2;\,
\R^3}+\int_{0}^{\infty}\lambda\;
\rmd\bigl(E_{-\lambda}\bfomega_a^+,\br \bfomega_a^+\bigr)_{2;\,
\R^3} \nonumber \\
&=\ \|A^{1/2}\bfomega_a^+\|_{2;\, \R^3}^2+ \int_{0}^{-a}\lambda\;
\rmd\bigl(E_{-\lambda}\bfomega_a^+,\br \bfomega_a^+\bigr)_{2;\,
\R^3} \nonumber \\
&=\ \|A^{1/2}\bfomega_a^+\|_{2;\, \R^3}^2+\int_{0}^{-a}(-\zeta)\;
\rmd\bigl(E_{\zeta}\bfomega_a^+,\br \bfomega_a^+\bigr)_{2;\,
\R^3}\ \leq\ \|A^{1/2}\bfomega_a^+\|_{2;\, \R^3}^2. \label{4.4}
\end{align}
We observe from (\ref{4.3}) and (\ref{4.4}) that for any value of
$a$, the second integral on the right hand of (\ref{4.1}) is less
than or equal to $\|A^{1/2}\bfomega_a^+\|_{2;\, \R^3}^2$. Thus,
applying also (\ref{4.2}), we obtain
\begin{equation}
\|A^{1/2}\bfomega^+\|_{2;\, \R^3}^2\ \leq\ \cc15\,
a_+^3+\|A^{1/2}\bfomega^+_a\|_{2;\, \R^3}^2. \label{4.5}
\end{equation}
Condition (iii) of Theorem \ref{T2} implies that the right hand
side of (\ref{4.5}) is integrable on the interval $(0,T)$. The
proof of Theorem \ref{T2} can now be completed in the same way as
the proof of Theorem \ref{T1} in the paragraph ``the case of
condition (i)'' in Section \ref{S2}.

Let us further assume that condition (iv) holds. Let us at first
suppose that $a\geq 0$, i.e.~$a=a_+$. In order to estimate
$\|A^{1/2}\bfomega^+_a\|_{2;\, \R^3}$, we can copy the proof of
Theorem \ref{T1} from ``the case of condition (ii)'' (which is now
replaced by ``the case of condition (iv)'') up to (\ref{2.14}); we
only consider $\bfomega^+_a$ instead of $\bfomega^+$ and
$\omega^+_{a3}$ instead of $\omega^+_3$. By analogy with
(\ref{2.14}), we obtain
\begin{equation}
\|A^{1/2}\bfomega^+_a\|_{2;\, \R^3}^2\ \leq\ \cc05\,
\|\bfomega^+_a\|_{2;\, \R^3}^2+\cc06\,
\|(-\Delta)^{3/4}\omega^+_{a3}\|_{2;\, \R^3}^2\br.
\label{4.6}
\end{equation}
Inequalities (\ref{4.5}) and (\ref{4.6}) yield
\begin{equation}
\|A^{1/2}\bfomega^+\|_{2;\, \R^3}^2\ \leq\ \cc15\, a_+^3+ \cc05\,
\|\bfomega^+_a\|_{2;\, \R^3}^2+\cc06\,
\|(-\Delta)^{3/4}\omega^+_{a3}\|_{2;\, \R^3}^2\br.
\label{4.7}
\end{equation}
Further, we suppose that $a<0$. Now, estimate (\ref{4.6}) is not
true due to this reason: the derivation of (\ref{4.6}) requires
the identity
\begin{equation}
(A\bfomega^+_a,\bfomega^+_a)_{2;\, \R^3}\ =\ (\curl\,
\bfomega^+_a,\bfomega^+_a)_{2;\, \R^3}, \label{4.10}
\end{equation}
analogous to the identity $(A\bfomega^+,\bfomega^+)_{2;\,
\R^3}=(\curl\, \bfomega^+,\bfomega^+)_{2;\, \R^3}$, which was used
in (\ref{2.3}) and which lead to (\ref{2.14}). However, while
(\ref{4.10}) holds in the case $a\geq 0$, it does not hold for
$a<0$ (which we now assume). Thus, we begin the estimation of
$\|A^{1/2}\bfomega^+\|_{2;\, \R^3}^2$ from (\ref{2.14}). In order
to estimate the term $\|(-\Delta)^{3/4}\omega_3^+\|_{2;\, \R^3}^2$
on the right hand side of (\ref{2.14}), we write
$\bfomega^+_a=\bfomega_{(a,0)}+\bfomega^+$. The same formula also
holds for the third components: $\,
\omega^+_{a3}=\omega_{(a,0),3}+\omega^+_3$. This yields
$\omega^+_3=\omega^+_{a3}-\omega_{(a,0),3}$ and
\begin{equation}
\|(-\Delta)^{3/4}\omega^+_3\|_{2;\, \R^3}^2\ \leq\
\|(-\Delta)^{3/4}\omega^+_{a3}\|_{2;\, \R^3}^2+
\|(-\Delta)^{3/4}\omega_{(a,0),3}\|_{2;\, \R^3}^2, \label{4.9}
\end{equation}
where
\begin{align*}
\|(- & \Delta)^{3/4}\omega_{(a,0),3}\|_{2;\, \R^3}^2\ \leq\
\int_0^{\infty}\rmd\bigl(F_{\lambda}A^{3/2}\bfomega_{(a,0)},\br
A^{3/2}\bfomega_{(a,0)}\bigr)_{2;\, \R^3} \nonumber \\
&=\ \int_0^{\infty}\lambda^3\, \rmd\bigl(F_{\lambda}
\bfomega_{(a,0)},\br\bfomega_{(a,0)}\bigr)_{2;\, \R^3}\ =\
\int_0^{\infty}\lambda^3\, \rmd\bigl((E_{\lambda}-E_{-\lambda})
\bfomega_{(a,0)},\br \bfomega_{(a,0)}\bigr)_{2;\, \R^3}\ \nonumber
\\ \noalign{\vskip 2pt}
&=\ -\int_0^{\infty}\lambda^3\, \rmd\bigl(E_{-\lambda}
\bfomega_{(a,0)},\br \bfomega_{(a,0)}\bigr)_{2;\, \R^3}.
\end{align*}
The last equality holds because
$E_{\lambda}\bfomega_{(a,0)}=\bfomega_{(a,0)}$ for $\lambda>0$,
which means that \\
$\rmd\bigl(E_{\lambda}\bfomega_{(a,0)},\br\bfomega_{(a,0)}\bigr)_{2;\,
\R^3}=0$. Further, we have
\begin{displaymath}
-\int_0^{\infty}\lambda^3\, \rmd\bigl(E_{-\lambda}
\bfomega_{(a,0)},\br \bfomega_{(a,0)}\bigr)_{2;\, \R^3}\ =\
-\int_0^{-a}\lambda^3\, \rmd\bigl(E_{-\lambda}
\bfomega_{(a,0)},\br \bfomega_{(a,0)}\bigr)_{2;\, \R^3}
\end{displaymath}
because $E_{-\lambda}\bfomega_{(a,0)}=\bfzero$ for $-\lambda<a$,
i.e. $\lambda>-a$. Using the substitution $\lambda=-\zeta$, the
last integral transforms to
\begin{equation}
\hspace{-10pt} -\int_0^{-a}(-\zeta)^3\,
\rmd\bigl(E_{\zeta}\bfomega_{(a,0)},\br
\bfomega_{(a,0)}\bigr)_{2;\, \R^3}\ =\ \int_0^{|a_-|}\zeta^5\,
\rmd\bigl(E_{\zeta}\bfv_{(a,0)},\br \bfv_{(a,0)}\bigr)_{2;\,
\R^3}\ \leq\ \cc15\, |a_-|^5. \label{4.8}
\end{equation}
Using now (\ref{2.14}), (\ref{4.9}) and (\ref{4.8}), we obtain the
inequality
\begin{equation}
\|A^{1/2}\bfomega^+\|_{2;\, \R^3}^2\ \leq\ \cc05\,
\|\bfomega^+\|_{2 ;\, \R^3}^2+\cc06\,
\|(-\Delta)^{3/4}\omega^+_{a3}\|_{2;\, \R^3}^2+\cc06\br\cc15\,
|a_-|^5.
\label{4.11}
\end{equation}
Both the right hand sides of (\ref{4.7}) and (\ref{4.11}) are
integrable, as functions of variable $t$, on the interval $(0,T)$
due to condition (iv) of Theorem \ref{T2}. The proof can now be
again finished in the same way as the proof of Theorem \ref{T1} in
Section \ref{S2}.

\section{Concluding remarks}
\label{S3}

\begin{remark}[the meaning of functions $\bfv^+$ and $\bfomega^+$]
\label{R1} \rm
Using the spectral resolution of identity $\{E_{\lambda}\}$
associated with operator $\curl$, we can express velocity $\bfv$
and the corresponding vorticity $\bfomega$ by the formulas
\vspace{-6pt}
\begin{equation}
\bfv=\int_{-\infty}^{\infty}\rmd E_{\lambda}(\bfv), \qquad
\bfomega=\int_{-\infty}^{\infty}\lambda\; \rmd E_{\lambda}(\bfv)=
\int_{-\infty}^{\infty}\rmd E_{\lambda}(\bfomega).
\label{3.1}
\end{equation}
In accordance with the heuristic understanding of the definite
integral, we can interpret the first integral in (\ref{3.1}) as a
sum of ``infinitely many'' contributions $\rmd E_{\lambda}(\bfv)$,
each of whose is an ``infinitely small'' Beltrami flow. (Recall
that {\it Beltrami flows} are flows, whose vorticity is parallel
to the velocity. Here, concretely, $\curl\, \rmd
E_{\lambda}(\bfv)=\lambda\; \rmd E_{\lambda}(\bfv)$.) Function
$\bfv^+$ can now be understood to be the sum of only those
``infinitely many'' ``infinitely small'' contributions, whose
vorticity is a positive multiple of velocity. (We call them the
{\it positive Beltrami flows.})
\end{remark}

\begin{remark}[flow in the neighbourhood of a singularity]
\label{R2} \rm
Theorem \ref{T1} is also true if $\bfomega^+$ (respectively
$\omega_3^+$) is replaced by $\bfomega^-$ (respectively
$\omega_3^-$). Thus, both the conditions (i) and (ii) show that if
weak solution $\bfv$ has a singular point then the singularity
must contemporarily develop in the ``positive part'' $\bfv^+$  of
function $\bfv$ (the contribution to $\bfv$ coming from the
positive Beltrami flows) as well as in the ``negative part''
$\bfv^-$ (the contribution from the negative Beltrami flows). The
singularity must even develop at the same spatial point. (This can
be proven by an appropriate localization procedure.)
\end{remark}

\begin{remark}[the role of large frequencies]
\label{R3} \rm
Suppose, for simplicity, that function $a$ considered in Theorem
\ref{T2} is positive. Then projection $P^+_{a}$ defined by
(\ref{G1}) can be interpreted as a reduction to the positive
Beltrami flows with ``high frequencies'', concretely the
frequencies comparable to $a$ and higher. Theorem \ref{T2} shows
that if a singularity develops in solution $\bfv$, then it must
especially develop in the part of $\bfv$ (respectively its
vorticity $\bfomega$) that consists of positive Beltrami flows
with the ``large'' frequencies (i.e.~$\sim a$ and higher). Since
the functions $a_+$, $\bfomega^+_a$ and $\omega_{a3}^+$ can be
replaced by $a_-$, $\bfomega^-_a$ and $\omega_{a3}^-$ in Theorem
\ref{T2}, the singularity must also develop in the part of $\bfv$
(respectively vorticity $\bfomega$) that consists of negative
Beltrami flows with ``large'' frequencies. The singularities must
appear in both the parts at the same space--time point.
\end{remark}

\begin{remark}
\label{R4} \rm
If function $a$ in Theorem \ref{T2} identically equals $-\infty$
in $(0,T)$ then $P^+_{a}=I$ and $\bfomega^+_a=\bfomega$ in
$(0,T)$. In this case, condition (iii) is the condition on the
whole vorticity $\bfomega$, and it requires that $\bfomega\in
L^2(0,T;\, D(S^{1/4}))$. (Recall that $S$ is the Stokes operator
in $\Ls$.) The space $D(S^{1/4})$ is continuously imbedded in
$\bfL^3(\R^3)$. Besides that, it is known that if $\bfomega\in
L^2(0,T;\, \bfL^3(\R^3))$ then solution $\bfv$ has no singular
points in $Q_T$, see e.g.~\cite{BdV}. This comparison (made for
$a\equiv -\infty$) gives hope that condition (iii) might be
perhaps generalized so that it would only require $\bfomega^+_a
\in L^2(0,T;\, \bfL^3(\R^3))$ instead of $\bfomega^+_a \in
L^2(0,T;\, D(S^{1/4}))$ also for other functions $a$. Similar
generalizations might also concern conditions (i), (ii) and (iv).
\end{remark}

%---------------------------------------------------------------
\vspace{4pt} \noindent
{\bf Acknowledgments.} \ The research was supported by the Grant
Agency of the Czech Republic (grant No.~201/08/0012), by the
Academy of Sciences of the Czech Republic (RVO 67985840) and by
the University Sud, Toulon--Var, Laboratoire SNC.

%----------------------------------------------------------------------

\noindent
{\it Author's addresses:} \nopagebreak

\bigskip \noindent
\begin{minipage}{62mm}
\leftline{Ji\v r\'\i\ Neustupa}
\leftline{Czech Academy of Sciences}
\leftline{Mathematical Institute}
\leftline{\v Zitn\'a 25, 115 67 Praha 1}
\leftline{Czech Republic}
\leftline{neustupa@math.cas.cz}
\end{minipage}
\kern 10mm
\begin{minipage}{62mm}
\leftline{Patrick Penel}
\leftline{Universit\'e du Sud, Toulon--Var}
\leftline{Dep.~Math\'ematique}
\leftline{BP 20132, 83957 La Garde}
\leftline{France}
\leftline{penel@univ-tln.fr}
\end{minipage}

\end{document}